\documentclass[letterpaper, 10 pt, conference]{ieeeconf}  

\IEEEoverridecommandlockouts                              
\usepackage{amsmath}
\usepackage{graphics} 
\usepackage{epsfig} 
\usepackage{mathptmx} 
\usepackage{times} 
\usepackage{amssymb}  
\usepackage{bm}
\usepackage{etoolbox}
\usepackage{hyperref}
\usepackage{lipsum}
\usepackage{booktabs}
\usepackage{color}
\hypersetup{
	colorlinks=true,
	linkcolor=black,
	urlcolor=blue,
	pdftitle={},
	pdfpagemode=FullScreen,
}
\usepackage{algorithmic}
\usepackage{stfloats}

\usepackage{afterpage}
\usepackage{amsmath}
\newcommand{\pd}[2]{\frac{\partial #1}{\partial #2}}

\usepackage{diagbox}
\usepackage{multirow} 
\usepackage{amsfonts}
\usepackage{graphicx}
\usepackage{xparse}
\usepackage{mathtools}
\usepackage{stmaryrd}
\usepackage[ruled,vlined]{algorithm2e}
\usepackage{caption}

\usepackage{subcaption}
\usepackage{tikz}
\usepackage{graphicx}
\usepackage{svg}
\usepackage{physics}

\usepackage{float}

\usepackage{etoolbox}

\newtheorem{theorem}{Theorem}

\newtheorem{assumption}[theorem]{Assumption}
\newtheorem{corollary}[theorem]{Corollary}
\newtheorem{proposition}[theorem]{Proposition}

\newtheorem{lemma}[theorem]{Lemma}

\title{\textbf{ An optimal-control framework for reaction diffusion systems with application to   synthetic developmental biology. }}

\author{\small
    Mohamed Amine Ouchdiri$^{1}$, Hamza Faquir$^{2}$, Saad Benjelloun$^{3}$,  Mohamed Maghenem$^{4}$, Irene Otero-Muras$^{2}$ and Adnane Saoud$^{1}$\\
    \thanks{$^{1}$Mohamed Amine Ouchdiri and Adnane Saoud are with College of Computing, University Mohammed VI Polytechnic, Benguerir, Morocco}
    \thanks{$^{2}$ Hamza Faquir and Irene Otero-Muras are with Institute for Integrative Systems Biology
(CSIC-UV), Spanish National Research Council, 46980 Valencia, Spain.
 }
    \thanks{$^{3}$Saad Benjelloun is with De Vinci Higher Education, De Vinci Research Center, Paris, France.}
    \thanks{$^{4}$ Mohamed Adlene Maghenem is with Université Grenoble Alpes, CNRS, Grenoble-INP, GIPSA-lab, Grenoble, France.}
}

\begin{document}

\maketitle
\thispagestyle{empty}
\pagestyle{empty}

\begin{abstract}
 Reaction-diffusion systems offer a powerful framework for understanding self-organized patterns in biological systems, yet controlling these patterns remains a significant challenge. As a consequence, we present a rigorous framework of optimal control for a class of  coupled reaction-diffusion systems. The  couplings are justified by the shared regulatory mechanisms encountered in synthetic biology. Furthermore, we introduce inputs and polynomial input-gain functions to guarantee well-posedness of the control system while maintaining  biological relevance. As a result, we formulate an optimal control problem and derive necessary optimality conditions. 
We demonstrate our framework on an instance of such equations  modeling the Nodal-Lefty interactions in mammalian cells.  Numerical simulations showcase  the effectiveness in directing pattern towards diverse targeted ones.  
\end{abstract}

\section{Introduction}
\label{sec1}
The study of optimal control for semilinear parabolic partial differential equations (PDE)s is a cornerstone in optimal control theory, providing essential tools for guiding and manipulating dynamical systems that evolve in both space and time. Our work introduces a novel mathematical framework of optimal control for a class of coupled reaction-diffusion systems. In fact, this work belongs to the broader contexts of optimal control in semilinear parabolic PDEs. The theoretical foundations for the control of parabolic PDEs date back to the seminal contribution of Jacques-Louis Lions~\cite{Lions1971}, who formulated fundamental principles regarding the existence of solutions, uniqueness properties and regularity conditions. This mathematical foundation was subsequently expanded in ~\cite{barbu1993}, incorporating state constraints and variational inequalities, substantially broadening the applicability of the theory. Further advancements are due to Fredi Tröltzsch~\cite{Troeltzsch2010}, who provided systematic approaches to derive necessary optimality conditions as well as efficient numerical algorithms to solve complex optimal control problems. The impact of these developments extends far beyond the theoretical interest, finding applications in diverse fields, including control of diffusion processes~\cite{Lions1971}, optimization
of chemical reaction kinetics in industrial
reactors~\cite{christofides2001}, and management of population dynamics
in ecological systems~\cite{clark1990}.

The considered class of 
reaction-diffusion systems is known to capture the interaction among multiple species as well as their emerging patterns in the context of synthetic developmental biology. 
For example, in \cite{Casas2018,Ryll2016},  nonlinear reaction--diffusion systems are shown to exhibit dynamical patterns such as  traveling waves, spiral waves, and moving spot patterns. 
Beyond these dynamical patterns, Turing patterns with their characteristic self-organising structures remain a central benchmark both theoretically and experimentally \cite{turing}.
In particular, the works in \cite{sekine,diambra} have shown that, in the context of two-species, spatial Turing patterns, such as stripes and  spots, can be formed. 

Upon generalizing the models in \cite{sekine,diambra}, we  allow for more sophisticated 
patterning capabilities, such as density-dependent patterns encountered in bacterial colonies \cite{karig} and oscillating patterns encountered in spatial gene expressions \cite{tica}. Despite this richness in terms of pattern generation, there is still a real need for control techniques to steer the system from one pattern to anther \cite{santos2019using}, \cite{ebrahimkhani2019synthetic}. Hence, we introduce an optimal-control framework to address this question. 
In particular, we provide mathematical guarantees of well-posedness for the control-to-state mapping, as well as necessary optimality conditions for the formulated optimal-control problem. 
To demonstrate the effectiveness of our theoretical results, we consider a relatively-simple control inputs guiding a complex biological system, known as the Nodal-Lefty reaction-diffusion system. The latter was experimentally reconstructed in mammalian cells by \cite{sekine}. Through the use of  linear input-gain functions, our framework successfully drive the biological system from an  initial pattern to diverse target ones. The implementability of our control strategy is justified by existing optogenetic technologies, enabling practical experimental validation. Examples include light-inducible systems developed in \cite{Kennedy2010} and  \cite{Levskaya2009}, as well as recent advances in photoactivatable Nodal receptors \cite{Sako2016} and spatial light patterning \cite{McNamara2023}.

The paper is organized as follows. Section \ref{sec3} presents the mathematical model. Section \ref{sec4} establishes well-posedness and control-to-state properties. Section \ref{sec5} derives the optimality conditions. Section \ref{sec6} applies the framework to the Nodal-Lefty system with numerical simulations. Section \ref{sec7} concludes. Due to space constraints, some proofs are omitted and will be published elsewhere.

\section{PRELIMINARIES AND MATHEMATICAL MODEL}
\label{sec3}

\subsection{Mathematical Model}
We consider a reaction-diffusion system with $n$ species on a bounded domain $\Omega \subset \mathbb{R}^d$ ($d \in \{1,2\}$) with Lipschitz boundary, over time interval $(0,T)$. Let $Q := \Omega \times (0,T)$ denote the space-time cylinder and $\Sigma := \partial\Omega \times (0,T)$ its lateral surface. The system dynamics are:
\begin{equation}
	\begin{cases}
		\pd{y}{t} = D\Delta y + \alpha(x)H(y) - \Gamma y + f(u) & \text{in } Q, \\
		\pd{y}{\nu} = 0 & \text{on } \Sigma, \\
		y(x,0) = y_0(x) & \text{for } x \in \Omega,
	\end{cases}
	\label{eq:main_system}
\end{equation}
where $y := (y_1, y_2, \ldots, y_n)^\top$ represents species concentrations, $u := (u_1, u_2, \ldots, u_n)^\top$ is the control vector with each component affecting its corresponding state, $D := \text{diag}(d_1, d_2, \ldots, d_n)$ with $d_i>0$ is the diffusion matrix, $\Delta y := (\Delta y_1, \Delta y_2, \ldots, \Delta y_n)^\top$ with $\Delta = \sum_{i=1}^{d} \frac{\partial^2}{\partial x_i^2}$ denotes the Laplacian operator, $\alpha(x) := (\alpha_1(x), \alpha_2(x),\ldots, \alpha_n(x))^\top \in (L^\infty(\Omega))^n$ with $\alpha_i(x)\geq 0$ represents spatially-dependent maximum production rates, $H$ is the regulatory function capturing species interactions, $\Gamma := \text{diag}(\gamma_1, \gamma_2, \ldots, \gamma_n)$ with $\gamma_i > 0$ is the degradation matrix, and $f(u) := (f_1(u_1), f_2(u_2), \ldots, f_n(u_n))^\top$ is the input-gain function. The homogeneous Neumann boundary condition $\pd{y}{\nu} = 0$ ensures zero flux across the boundary, where $\frac{\partial}{\partial \nu}$ denotes the outward normal derivative.

\subsection{Function Spaces}
We employ standard Sobolev spaces for parabolic problems. Let $V := H^1(\Omega)$ denote the Sobolev space with norm $\|v\|_{V} := \left(\|v\|_{L^2(\Omega)}^2 + \|\nabla v\|_{(L^2(\Omega))^d}^2\right)^{1/2}$. The space $W(0,T) := \left\{ v \in L^2(0,T;V) \,\middle|\, \frac{\partial v}{\partial t} \in L^2(0,T;V^*) \right\}$, where $L^2(0,T;V)$ is the space of time-dependent functions in $V$ with finite energy over time, and $V^*$ is the dual space of $V$, allowing weak time derivatives. The space $W(0,T)$ is equipped with the norm $\|v\|_{W(0,T)} := \|v\|_{L^2(0,T;V)} + \left\|\frac{\partial v}{\partial t}\right\|_{L^2(0,T;V^*)}$. The solution space is defined as $Y := W(0,T)^n \cap L^\infty(Q)^n$ with norm $\|y\|_{Y} := \|y\|_{W(0,T)^n} + \|y\|_{L^\infty(Q)^n}$. We assume the initial condition $y_0 \in (L^\infty(\Omega))^n$ for additional regularity. The control space is defined as $U = (L^\infty(Q))^n$.

\subsection{Regulatory-Function Properties}
The regulatory function $H: \mathbb{R}^n_+ \to [0,1]$ is a $C^\infty$ function taking the general form:
\begin{equation}
	H(\mathbf{y}) := \frac{A(\mathbf{y})}{A(\mathbf{y}) + B(\mathbf{y})},
\end{equation}
where:
\begin{itemize}
	\item $A(\mathbf{y}) := \left(\sum_{i \in \mathcal{A}} w_i y_i^{n_i}\right)^m$ represents the activation term.
	\item $B(\mathbf{y}) := K^m \cdot \left(1 + \sum_{j \in \mathcal{I}} \left(\frac{y_j}{K_j}\right)^{n_j}\right)^m$ represents the inhibition term.
\end{itemize}

Here, $\mathcal{A}$ and $\mathcal{I}$ are the sets of activator and inhibitor species indices, respectively; $w_i > 0$ are weighting coefficients; $n_i > 0$ are Hill coefficients; $m > 0$ is an overall cooperativity coefficient; $K > 0$ is the half-maximal activation constant; and $K_j > 0$ are inhibition constants. 

The following proposition establishes a crucial regularity property of the regulatory function that ensures the well-posedness of our control framework.
\begin{proposition}
\label{prop:bounded_derivatives}
Assuming that $n_i \geq 1$ for all $i \in \{1,2,\ldots,n\}$, $m\geq 1$ and the state variables remain positive ($y_i \geq 0$ for all $i$), then the first derivatives of the regulatory function $H$ are uniformly bounded. That is, there exists a constant $C > 0$ such that
\begin{equation}
\left|\frac{\partial H}{\partial y_k}(\mathbf{y})\right| \leq C \qquad \forall k \in \{1,2,\ldots,n\}, \quad \forall \mathbf{y} \in \mathbb{R}^n_+.
\end{equation}
\end{proposition}

\subsection{Input Function Properties}
The input function $f: \mathbb{R}^n_+ \to \mathbb{R}^n$ is defined as 
$f(u) := (f_1(u_1), f_2(u_2), \ldots, f_n(u_n))^\top$ where each component 
$f_i: \mathbb{R}^+ \to \mathbb{R}$ satisfies the following assumption.

\begin{assumption}
	\label{assume:input}
    $f_i \in C^{\infty}(\mathbb{R}^+)$ (infinitely differentiable) and 
	\begin{enumerate}
		\item $f_i(u_i) \geq 0$ for all $u_i \in \mathbb{R}^+$;
		\item $f_i$ is  Lipschitz continuous on $\mathbb{R}^+$;
		\item $f_i$ exhibits polynomial growth, i.e., there exist constants $C > 0$ and 
		$m \in \mathbb{N}$ such that
        \begin{equation}
		|f_i(u_i)| \leq C(1 + |u_i|^m) \quad \forall u_i \in \mathbb{R}^+.    
		\end{equation}
		
	\end{enumerate}
\end{assumption}

The polynomial-growth condition ensures solution existence and boundedness \cite{pierre2010global}, while capturing a wide range of biologically-relevant control mechanisms.

\section{Analysis of Well-Posedness and Control-to-State Properties}
\label{sec4}
First, we study the well-posedness to ensure that unique solutions exist for given controls and initial conditions. We then analyze the control-to-state mapping $S$, examining its Lipschitz continuity and differentiability properties essential for optimization algorithms. These intermidiate results are key for the necessary optimality conditions.

\begin{theorem}
	\label{th:wellposedness}
	For any control $u \in (L^\infty(Q))^n$, there exists a unique weak solution $y \in Y$ to the reaction-diffusion system \eqref{eq:main_system}. Moreover, the following 
	estimates hold.
		\begin{align}
			\|y\|_{W(0,T)^n} & \leq C_1(\|y_0\|_{(L^2(\Omega))^n} + \|\alpha\|_{(L^\infty(\Omega))^n} + \|u\|_{(L^\infty(Q))^n}^m),
		\\
			\|y_i\|_{L^\infty(Q)} & \leq \|y_{0,i}\|_{L^\infty(\Omega)} + \frac{\|\alpha_i\|_{L^\infty(\Omega)} + \|f_i(u_i)\|_{L^\infty(Q)}}{\gamma_i},
		\end{align}
	for each component $i \in \{1, 2, \ldots, n\}$, 
	where $C_1$ depends only on $\Omega$, $T$, $D$, and $\Gamma$.
\end{theorem}
The well-posedness proof employs a fixed-point argument via Schauder's theorem, constructing solutions through a compact operator on $L^2(0,T;L^2(\Omega))^n$ and establishing $L^\infty$ bounds via truncation methods.
\begin{proposition}
	\label{prop:control_to_state}
	Let $u \in (L^\infty(Q))^n$, and consider the control-to-state mapping 
	$S: (L^\infty(Q))^n \to Y$ defined by $S(u) = y$, where $y$ is the 
	unique weak solution to \eqref{eq:main_system}. The mapping $S$ is Lipschitz continuous, i.e., there exists $L > 0$ such that, for any controls $u^1, u^2 \in (L^\infty(Q))^n$ with 
	corresponding states $y^1 = S(u^1)$ and $y^2 = S(u^2)$, 
	\begin{equation}
		\|S(u^1) - S(u^2)\|_{Y} \leq L \|u^1 - u^2\|_{(L^\infty(Q))^n}.
	\end{equation}
\end{proposition}
 The Lipschitz continuity follows from linearization around the state difference and application of parabolic regularity theory. 
We now investigate the differentiability properties of the control-to-state operator. 

\begin{proposition}\label{differentiability}
The control-to-state operator $S$ is twice continuously Fr\'echet-differentiable, i.e., $S$ has first and second derivatives $S'(u)$ and $S''(u)$ at each $u \in (L^{\infty}(Q))^n$, with both derivatives depending continuously on $u$.
\end{proposition}
\begin{proof}
We apply the implicit function theorem. Let $G_Q: [L^\infty(Q)]^n \to Y$ be the solution operator mapping source terms to solutions of the linear parabolic system $\partial_t w - D\Delta w + \Gamma w = v$ with homogeneous boundary and initial conditions. Define $F: Y \times (L^\infty(Q))^n \to Y$ by
$$F(y,u) = y - G_Q[\alpha(x)H(y) + f(u)] - G_0[y_0]$$
where $G_0$ maps initial data to solutions. Since $H, f \in C^\infty$, they generate twice differentiable Nemytskii operators in $L^\infty(Q)$\cite{Troeltzsch2010}. The composition with the linear operator $G_Q$ preserves differentiability.

For invertibility of $F_y(y,u)$, we need to solve $F_y(y,u)w = z$ for any $z \in Y$. This reduces to the linear parabolic system:
$$\partial_t w - D\Delta w + \Gamma w - \alpha(x)(\nabla H(y) \cdot w) = g$$
with appropriate data. Since $|\nabla H|$ is bounded and $\Gamma > 0$, standard parabolic theory guarantees unique solvability. The implicit function theorem then ensures $S$ is twice continuously Fr\'echet-differentiable.
\end{proof}

\begin{lemma}
\label{cor:first_derivative}
	Let $(\bar{y}, \bar{u})$ be a solution to  \eqref{eq:main_system} with $\bar{u} \in (L^\infty(Q))^n$. 
    Then, for any direction $h \in (L^\infty(Q))^n$, the first derivative $S'(\bar{u})h = w$ is the unique solution to the linearized system
	\begin{equation}
		\begin{cases}
			\pd{w}{t} = D\Delta w + \alpha(x)(\nabla H(\bar{y}) \cdot w) - \Gamma w + J_f(\bar{u})h & \text{in } Q, \\
			\pd{w}{\nu} = 0 \quad  \quad \text{on } \partial\Omega \times (0,T), \\
			w(x,0) = 0 \quad  \text{for } x \in \Omega,
		\end{cases}
		\label{eq:first_derivative_system}
	\end{equation}
	where $\nabla H(\bar{y})$ is the gradient of $H$ at $\bar{y}$ and $J_f(\bar{u})$ is the diagonal Jacobian matrix of $f$ evaluated at $\bar{u}$.
\end{lemma}
\begin{proof}
By the implicit function theorem, $S'(\bar{u})h = -[F_y(\bar{y},\bar{u})]^{-1}F_u(\bar{y},\bar{u})h$. Computing $F_u(\bar{y},\bar{u})h = -G_Q[J_f(\bar{u})h]$ and applying $F_y(\bar{y},\bar{u})w = -F_u(\bar{y},\bar{u})h$, we obtain that $w = S'(\bar{u})h$ solves the linearized system \eqref{eq:first_derivative_system}.
\end{proof}
\section{Optimal Control Problem}
\label{sec5}
In this section, we formulate and analyze the 
optimal control problem for the reaction-diffusion system \eqref{eq:main_system}. 

We seek optimal control inputs that drive the system to a desired spatial pattern at time $T$ while minimizing the control effort.

Let $y_\Omega \in (L^\infty(\Omega))^n$ denote the desired-target state. The performance of the system is measured by the cost functional
\begin{equation}
	J(y,u) := \frac{\mu}{2}\int_\Omega \|y(x,T)-y_\Omega(x)\|^2\,dx+\frac{\lambda}{2}\int_Q \|u(x,t)\|^2\,dxdt,
	\label{eq:objective_functional}
\end{equation}
with weighting parameters $\mu> 0$ and $\lambda > 0$.

The admissible control functions are defined by
\begin{equation}
U_{ad} = \{u \in (L^\infty(Q))^n: u_a(x,t) \le u(x,t) \le u_b(x,t) \text{ a.e. in } Q\}
\end{equation}
with $u_a, u_b \in (L^\infty(Q))^n$ and $u_{a,i}(x,t) \ge 0$ for all $i = 1,\ldots,n$.

Thus, the optimal control problem is formulated as
\begin{equation}
	\begin{aligned}
		\min_{u\in U_{ad}} \quad & J(y,u), 
		\text{ subject to}  \text{ \eqref{eq:main_system}}.
	\end{aligned}
	\label{eq:optimal_control_problem}
\end{equation}

\begin{theorem}
\label{thm:existence}
	For the optimal control problem \eqref{eq:optimal_control_problem}, there exists at least one solution $u^* \in U_{ad}$ such that
	\begin{equation}
		J(S(u^*),u^*) = \min_{u \in U_{ad}} J(S(u),u),
	\end{equation}
	where $S: (L^\infty(Q))^n \to Y$ is the control-to-state mapping.
\end{theorem}
\begin{proof}
Let $\{u^k\} \subset U_{ad}$ be a minimizing sequence for $J$. Since $U_{ad}$ is bounded in $(L^\infty(Q))^n$, by Banach-Alaoglu there exists a subsequence (still denoted $\{u^k\}$) and $u^* \in U_{ad}$ such that $u^k \rightharpoonup u^*$ weakly-* in $(L^\infty(Q))^n$. The corresponding states $y^k = S(u^k)$ are uniformly bounded in $Y$ by Theorem 3. By Aubin-Lions compactness, extracting a further subsequence, $y^k \to y^*$ strongly in $L^2(Q)^n$ for some $y^*$. Passing to the limit in the weak formulation shows $y^* = S(u^*)$. The cost functional satisfies
$$J(S(u^*),u^*) \leq \liminf_{k \to \infty} J(S(u^k),u^k)$$
by weak-* lower semicontinuity of $\|u\|^2$ and strong convergence of $y^k(T)$ in $L^2(\Omega)^n$. Thus $u^*$ is optimal.
\end{proof}
We now establish the explicit form of the first derivative of the reduced-cost function, which characterizes the sensitivity of the system's response to infinitesimal perturbations in the control input.

\begin{proposition}
\label{prop:reduced_cost_derivative}
	Let $j(u) := J(S(u), u)$ be the reduced cost functional. For any $\bar{u} \in U_{ad}$ with the corresponding state $\bar{y} = S(\bar{u})$, the directional derivative of $j$ at $\bar{u}$ in the direction $h \in (L^\infty(Q))^n$ is given by
	\begin{equation}
		j'(\bar{u})(h) = \int_Q \left[ \lambda \bar{u} + f'(\bar{u})^T p \right] \cdot h \, dxdt,
	\end{equation}
	where $p$ is the adjoint state solving the backward system
	\begin{equation}
		\begin{cases}
			-\frac{\partial p}{\partial t} - D \Delta p + \Gamma p - \alpha(x) \nabla H(\bar{y})^T p = 0 \quad &\text{in } Q, \\
			\frac{\partial p}{\partial \nu} = 0 \quad &\text{on } \Sigma, \\
			p(x,T) = \mu(\bar{y}(x,T) - y_\Omega(x)) \quad &\text{in } \Omega.
		\end{cases}
        \label{eq:adjoint}
	\end{equation}
\end{proposition}
\begin{proof}
The derivative of $j$ at $\bar{u}$ in direction $h$ is:
$$j'(\bar{u})(h) = \mu \int_\Omega (\bar{y}(x,T) - y_\Omega(x)) \cdot w(x,T) \, dx + \lambda \int_Q \bar{u} \cdot h \, dxdt$$
where $w = S'(\bar{u})h$ solves the linearized system from Corollary \ref{cor:first_derivative}.

To eliminate the $w(x,T)$ term, introduce the adjoint state $p$ solving \eqref{eq:adjoint}. Multiplying the adjoint equation by $w$ and integrating over $Q$:
$$0 = \int_Q \left[-\frac{\partial p}{\partial t} - D \Delta p + \Gamma p - \alpha(x) \nabla H(\bar{y})^T p \right] \cdot w \, dxdt.$$

Integration by parts in time (using $w(x,0) = 0$) and space (using Neumann conditions) yields:
\begin{align}
0 &= \int_Q p \cdot \left[ \frac{\partial w}{\partial t} - D \Delta w - \alpha(x) \nabla H(\bar{y}) \cdot w + \Gamma w \right] dxdt \nonumber\\
&\quad - \int_\Omega p(x,T) \cdot w(x,T) \, dx
\end{align}

Substituting the linearized equation for $w$ and the terminal condition $p(x,T) = \mu(\bar{y}(x,T) - y_\Omega(x))$:
$$\int_\Omega \mu(\bar{y}(x,T) - y_\Omega(x)) \cdot w(x,T) \, dx = \int_Q p \cdot f'(\bar{u})h \, dxdt.$$

Therefore:
$$j'(\bar{u})(h) = \int_Q \left[ \lambda \bar{u} + f'(\bar{u})^T p \right] \cdot h \, dxdt.$$
\end{proof}

The following theorem provides the key characterization of optimal controls through a variational inequality, which serves as the foundation for the numerical implementation presented in Section V.
\begin{theorem}[First-order necessary optimality condition]
	\label{thm:first_order_optimality}
	Let $\bar{u} \in U_{ad}$ be a locally-optimal control for the problem \eqref{eq:optimal_control_problem},  with the corresponding state $\bar{y} = S(\bar{u})$. Then, $\bar{u}$ satisfies the variational inequality
	\begin{equation}
		\int_Q \left( \lambda \bar{u} + f'(\bar{u})^\top p \right) \cdot (u - \bar{u}) \, dxdt \geq 0 \quad \forall u \in U_{ad},
	\end{equation}
	where $p$ is the adjoint state described in Proposition \ref{prop:reduced_cost_derivative}.
\end{theorem}

\section{Application to Synthetic developmental biology}
\label{sec6}
\subsection{The Nodal-Lefty Model}
The Nodal-Lefty reaction-diffusion system, experimentally realized by Sekine et al. \cite{sekine}, motivates our control framework as it exemplifies the challenge of directing self-organized patterns in synthetic biology. Such a system represents a classic activator-inhibitor mechanism where Nodal (activator) promotes both its own production and that of Lefty (inhibitor), while Lefty inhibits Nodal activity. The key feature enabling pattern formation is the significant difference in diffusion rates, Nodal diffuses locally while Lefty spreads more widely \cite{muller2012differential}. 

The Nodal-Lefty dynamics are governed by

\begin{equation}
	\begin{cases}
    \begin{aligned}
		\frac{\partial y_n}{\partial t} &= \alpha_{n}(x) H(y_n,y_l) + \beta_{n}u_n - \gamma_{n} y_n + D_{n} \Delta y_n & \text{in } Q,\\
		\frac{\partial y_l}{\partial t} &= \alpha_{l}(x)H(y_n,y_l) + \beta_{l}u_l - \gamma_{l} y_l + D_{l} \Delta y_l, & \text{in } Q, \\
		\frac{\partial y_n}{\partial \nu} &= \frac{\partial y_l}{\partial \nu} = 0, & \text{on } \Sigma, \\
		y_n(x, 0) &= y_{n,0}(x), \quad y_l(x, 0) = y_{l,0}(x), & \text{in } \Omega.
    \end{aligned}
	\end{cases}
    \label{RD_nodal}
\end{equation}

Here, $y_n$ and $y_l$ represent Nodal and Lefty concentrations, respectively. The parameters $\alpha_n$, $\alpha_l$, $\beta_n$, and $\beta_l$ are the maximum production rates and control-input effects (nM·min$^{-1}$). The regulatory function models competitive inhibition:

\begin{equation}
	H(y_n,y_l) := \dfrac{y_n^{n_n}}{y_n^{n_n}+\left[k_n\left\{1+\left(\frac{y_l}{k_l}\right)^{n_l}\right\}\right]^{n_n}}.
\end{equation}

The system parameters taken from  \cite{sekine} are: $D_{n} = 1.96$ $\mu\text{m}^2\text{min}^{-1}$, $D_{l} = 56.39$ $\mu\text{m}^2\text{min}^{-1}$, $\gamma_{n} = 2.37 \times 10^{-3}$ $\text{min}^{-1}$, $\gamma_{l} = 5.65 \times 10^{-3}$ $\text{min}^{-1}$, $n_{n} = 2.63$, $n_{l} = 1.09$, $k_{n} = 9.28$ nM, $k_{l} = 14.96$ nM.

\subsection{Optimal Control Problem Formulation}
We formulate the optimal control problem for the Nodal-Lefty system as:

\begin{equation}
	\begin{aligned}
		\min_{u \in U_{ad}} J[y_n, y_l, u] &= \dfrac{\mu}{2}\int_{\Omega}(y_n(x, T)-y_{n,\Omega}(x))^2 dx \\
		&+ \dfrac{\mu}{2}\int_{\Omega}(y_l(x, T)-y_{l,\Omega}(x))^2 dx \\
		&+ \dfrac{\lambda}{2}\int_0^T \int_{\Omega} (u_n(x, t)^2 + u_l(x, t)^2) dx dt,
	\end{aligned}
\end{equation}

subject to the state equations \eqref{RD_nodal} and the admissible control set:
\begin{equation*}
	U_{ad} = \{u \in (L^{\infty}(Q))^2 \mid u_a \leq u(x, t) \leq u_b \text{ a.e. in } Q\},
\end{equation*}
where $u_a = (0,0)^T$ and $u_b = (1,1)^T$.

The desired states $(y_{n,\Omega}, y_{l,\Omega})$ correspond to different Turing patterns generated by varying $\alpha_n$ and $\alpha_l$. Our control objective consists in guiding the system from one stable pattern configuration to another target pattern by modulating the dynamical spatial distribution of the inputs $(u_n, u_l)$.

\subsection{Optimality Conditions and Numerical Results}

\subsubsection{Optimality Conditions}

We establish the necessary optimality conditions specific to the Nodal-Lefty system based on Theorem \ref{thm:first_order_optimality}.

\begin{corollary}
	Let $\bar{u} = (\bar{u}_n, \bar{u}_l) \in U_{ad}$ be an optimal control for the Nodal-Lefty system with corresponding states $\bar{y}_n$ and $\bar{y}_l$. Then $\bar{u}$ satisfies the variational inequality
	\begin{equation}
		\int_Q \left[ \lambda \bar{u}_n + \beta_n p_n \right] \cdot (u_n - \bar{u}_n) + \left[ \lambda \bar{u}_l + \beta_l p_l \right] \cdot (u_l - \bar{u}_l) \, dxdt \geq 0
	\end{equation}
	for all $u := (u_n, u_l) \in U_{ad}$, where $(p_n, p_l)$ is the adjoint state satisfying the backward system:
	\begin{equation}
		\begin{cases}
			-\frac{\partial p_n}{\partial t} - D_n \Delta p_n + \gamma_n p_n - \alpha_n \frac{\partial H}{\partial y_n}p_n - \alpha_l \frac{\partial H}{\partial y_n}p_l = 0, & \text{in } Q, \\
			-\frac{\partial p_l}{\partial t} - D_l \Delta p_l + \gamma_l p_l - \alpha_n \frac{\partial H}{\partial y_l}p_n - \alpha_l \frac{\partial H}{\partial y_l}p_l = 0, & \text{in } Q, \\
			\frac{\partial p_n}{\partial \nu} = \frac{\partial p_l}{\partial \nu} = 0, & \text{on } \Sigma, \\
			p_n(x,T) = \mu(\bar{y}_n(x,T) - y_{n,\Omega}(x)), & \text{in } \Omega, \\
			p_l(x,T) = \mu(\bar{y}_l(x,T) - y_{l,\Omega}(x)), & \text{in } \Omega.
		\end{cases}
	\end{equation}
\end{corollary}

\subsubsection{Numerical Simulation}

We implemented our optimal control framework using the Crank–Nicolson discretization, and solved the resulting optimization problem using the nonlinear conjugate gradient method with Polak-Ribière formula \cite{polak1969note}. The initial condition consists of a random spatial distribution with low and high concentration domains. We set our initial configuration to generate striped patterns ($\alpha_n=0.8$, $\alpha_l=4.0$) and investigate transitions to diverse targets. The simulations used a spatial domain $\Omega = [0, 800] \times [0, 800]$ $\mu\text{m}$ with grid resolution $\Delta x = \Delta y = 10$ $\mu\text{m}$ and time step $\Delta t = 0.5$ hours.

The control parameters remained fixed: $\beta_n = 0.8$, $\beta_l = 4.0$, regularization parameter $\lambda = 10^{-12}$, and tracking coefficient $\mu = 1.0$. Convergence occurred within 20-40 iterations for all cases.

\begin{figure*}
\centering
\begin{subfigure}[b]{0.48\textwidth}
    \centering
    \includegraphics[width=\textwidth]{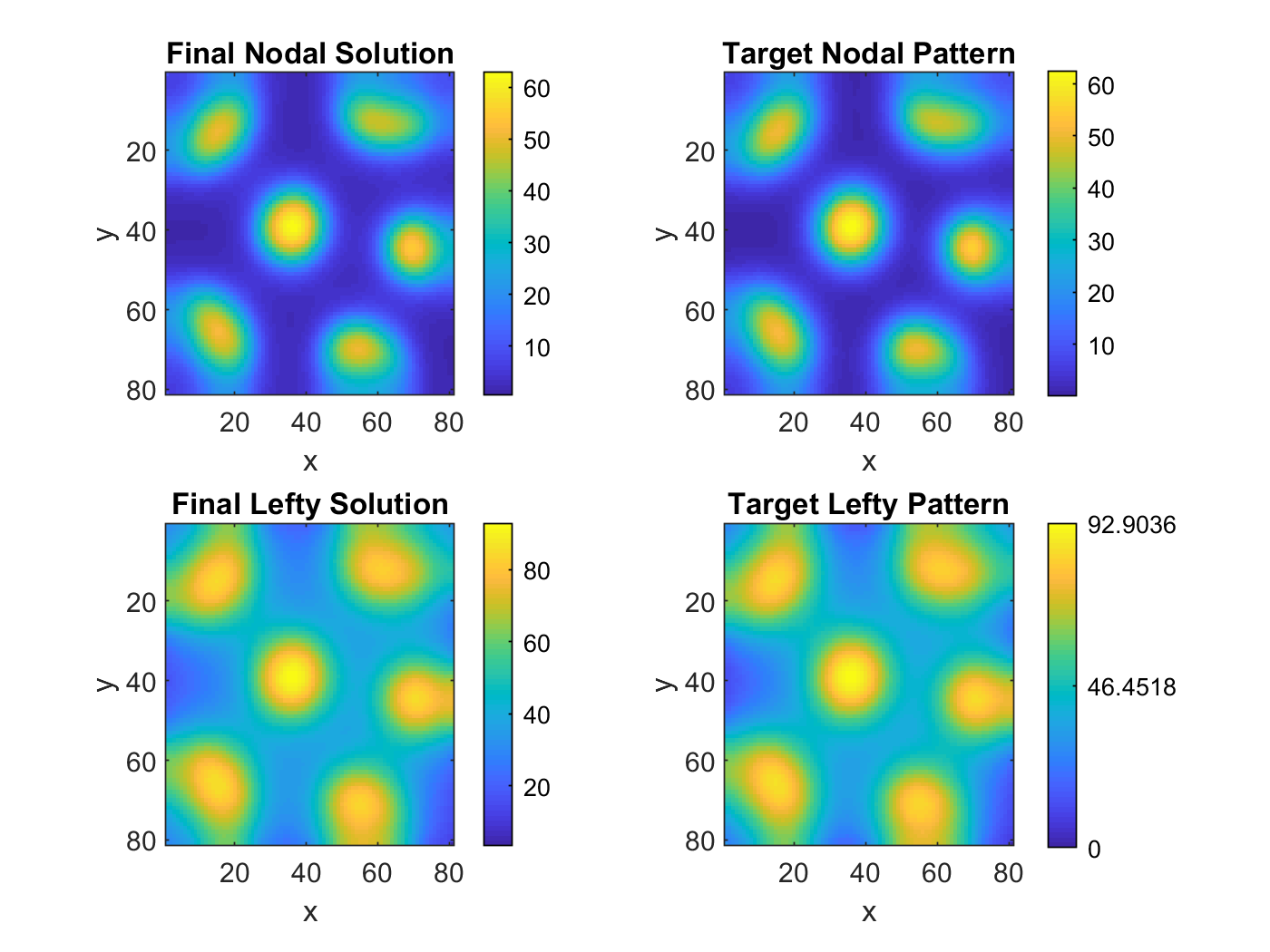}
    \caption{Case 1: Spotted pattern}
\end{subfigure}
\hfill
\begin{subfigure}[b]{0.48\textwidth}
    \centering
    \includegraphics[width=\textwidth]{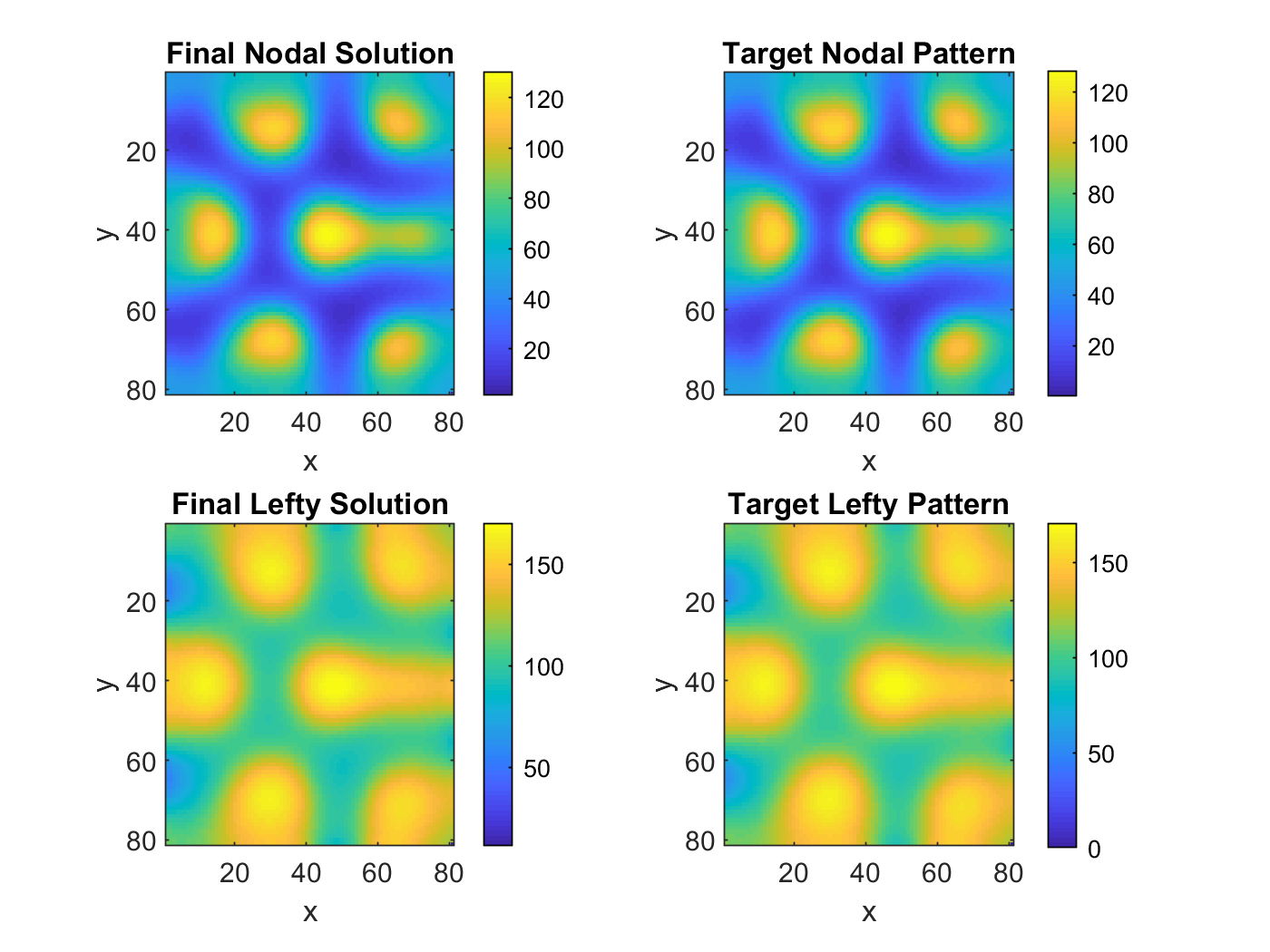}
    \caption{Case 2: Modified spotted}
\end{subfigure}
\vspace{0.3cm}
\begin{subfigure}[b]{0.48\textwidth}
    \centering
    \includegraphics[width=\textwidth]{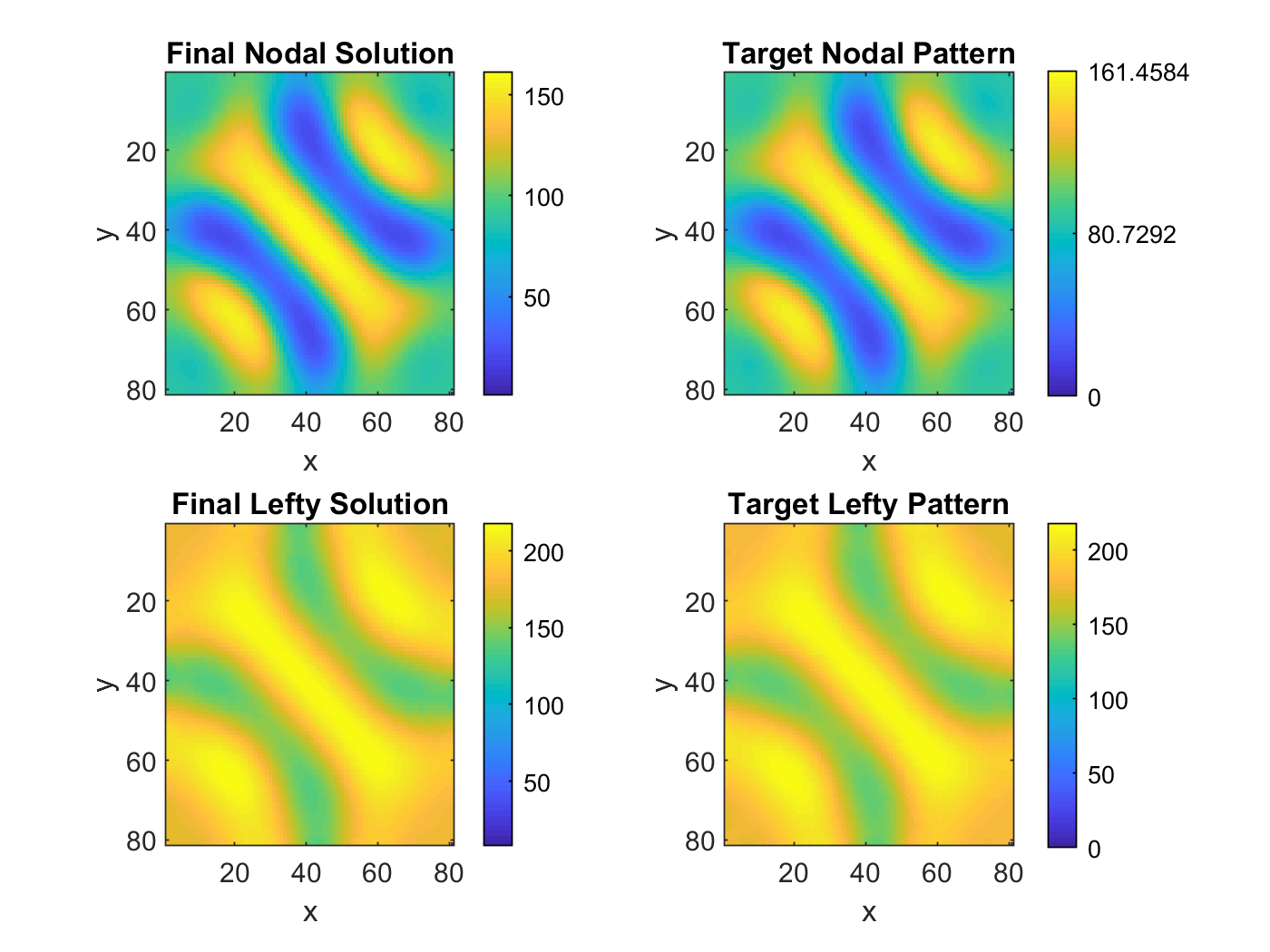}
    \caption{Case 3: Modified striped}
\end{subfigure}
\hfill
\begin{subfigure}[b]{0.48\textwidth}
    \centering
    \includegraphics[width=\textwidth]{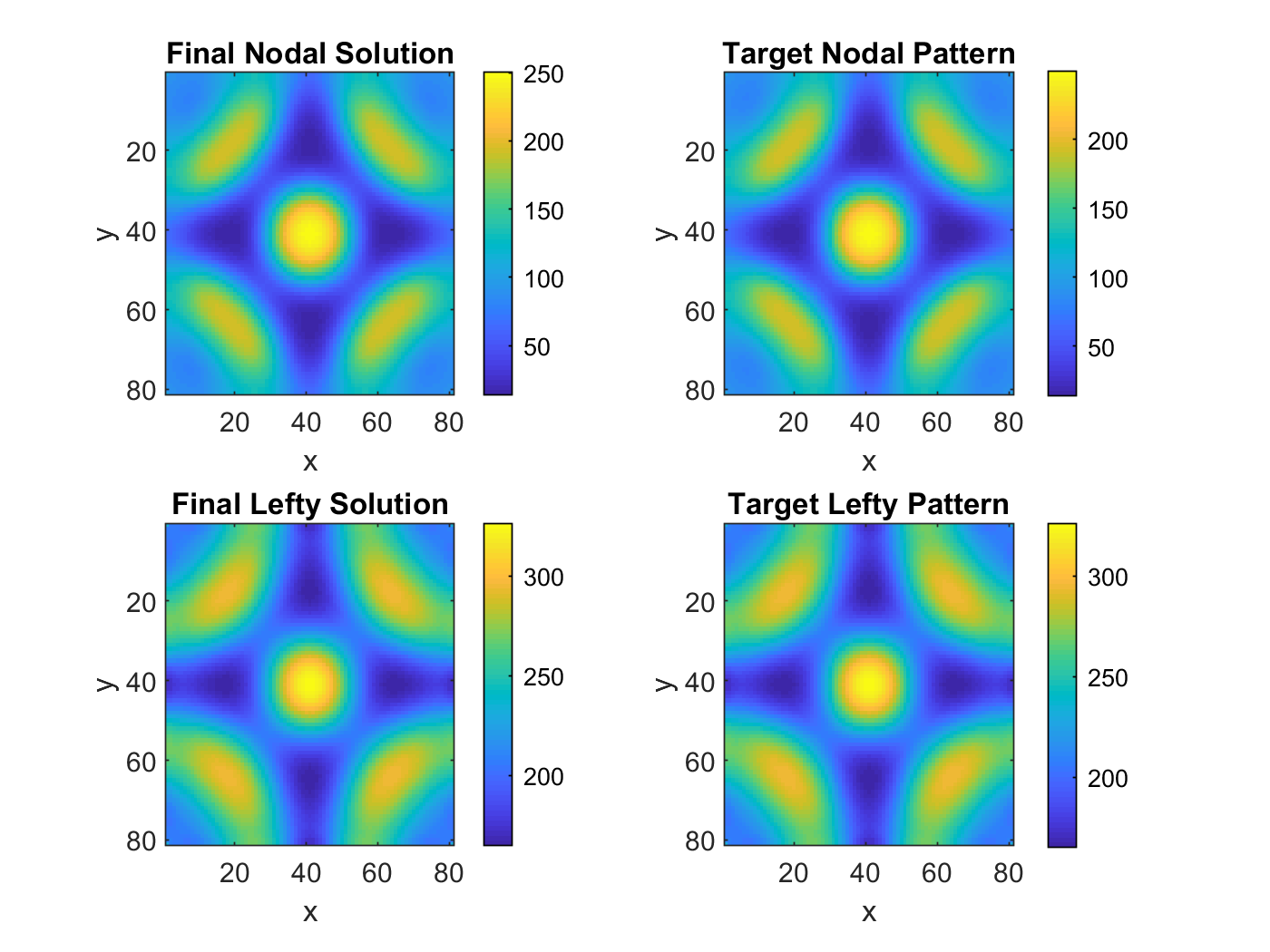}
    \caption{Case 4: Striped/Radial}
\end{subfigure}
\caption{Final solution (left) and target pattern (right) comparisons for Nodal (top) and Lefty (bottom) concentrations in each case. Blue: low, yellow: high concentration (nM). }
\label{fig:pattern_comparisons}
\end{figure*}
\begin{figure*}
\centering
\begin{subfigure}[b]{0.48\textwidth}
    \centering
    \includegraphics[width=\textwidth]{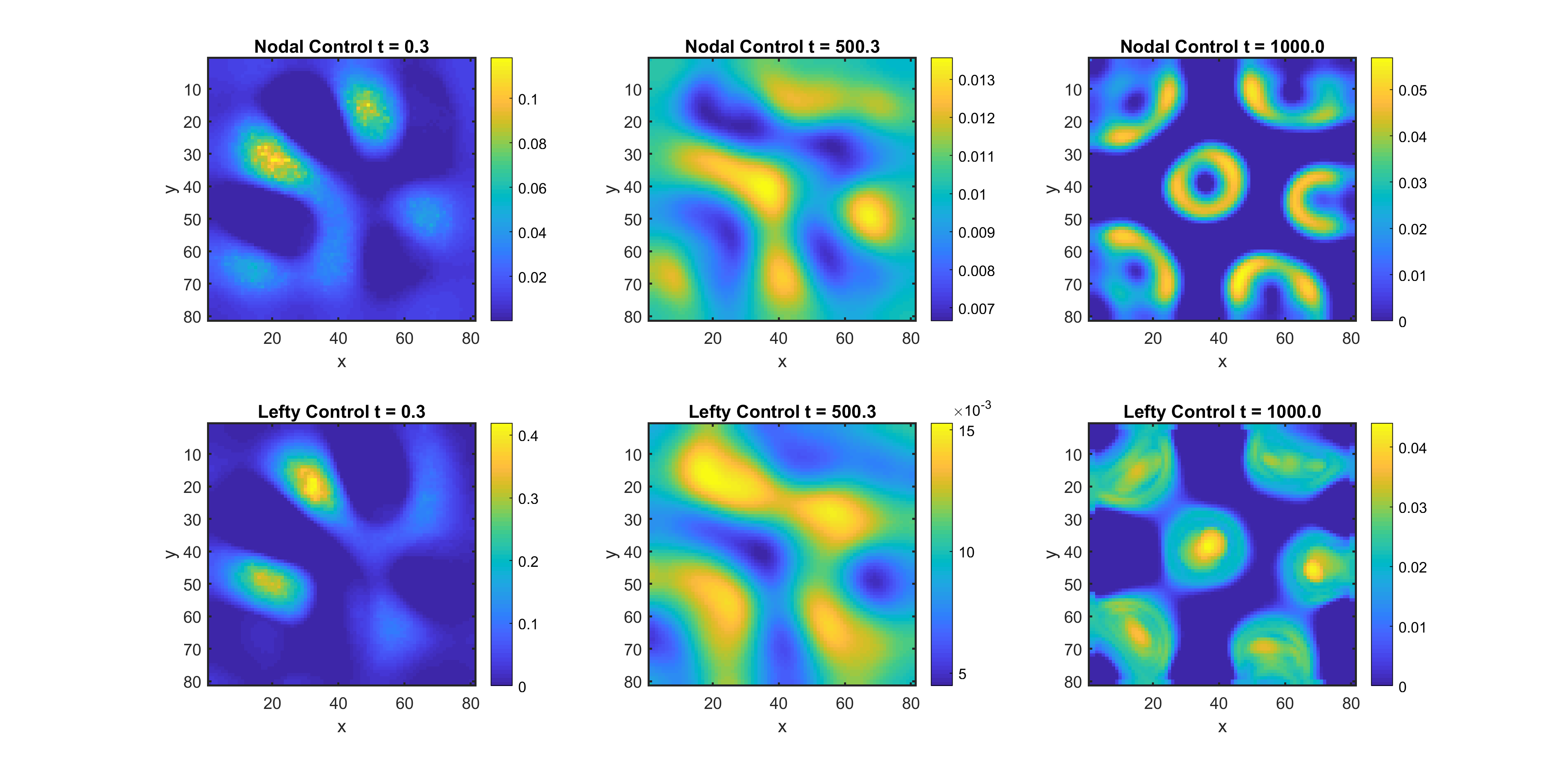}
    \caption{Case 1: Spotted formation}
\end{subfigure}
\hfill
\begin{subfigure}[b]{0.48\textwidth}
    \centering
    \includegraphics[width=\textwidth]{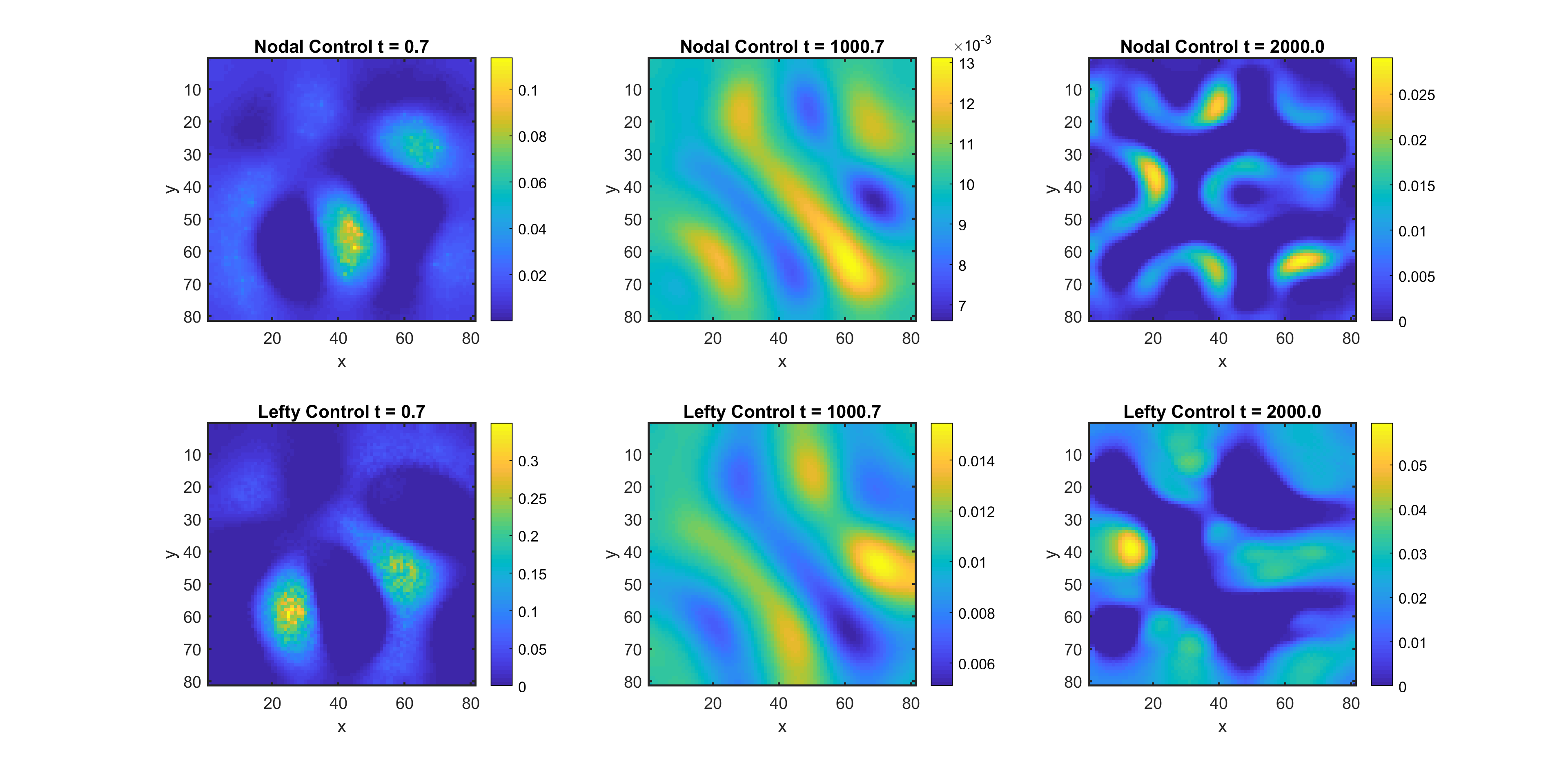}
    \caption{Case 2: Modified spotted}
\end{subfigure}
\vspace{0.3cm}
\begin{subfigure}[b]{0.48\textwidth}
    \centering
    \includegraphics[width=\textwidth]{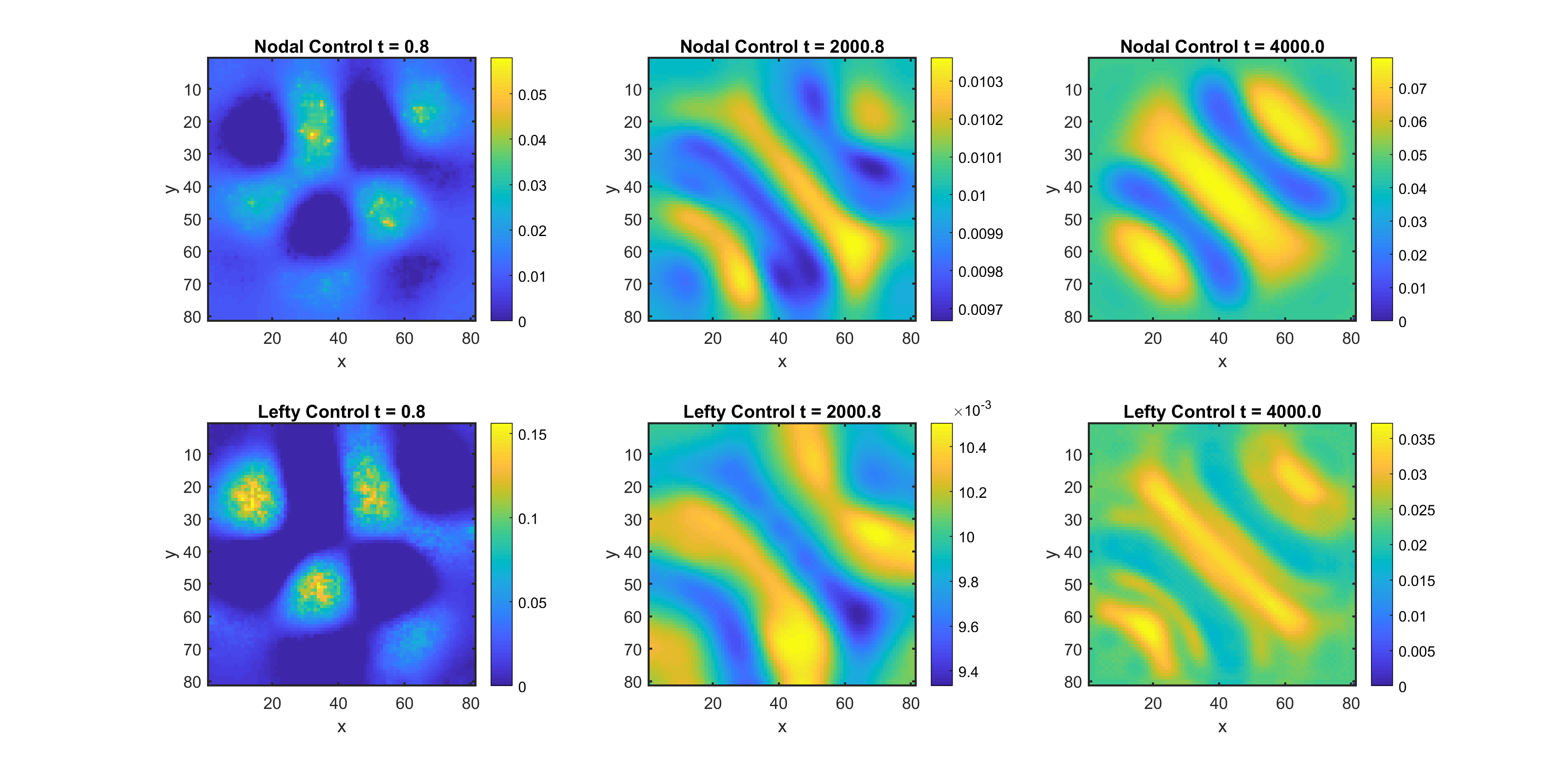}
    \caption{Case 3: Stripe redirection}
\end{subfigure}
\hfill
\begin{subfigure}[b]{0.5\textwidth}
    \centering
    \includegraphics[width=\textwidth]{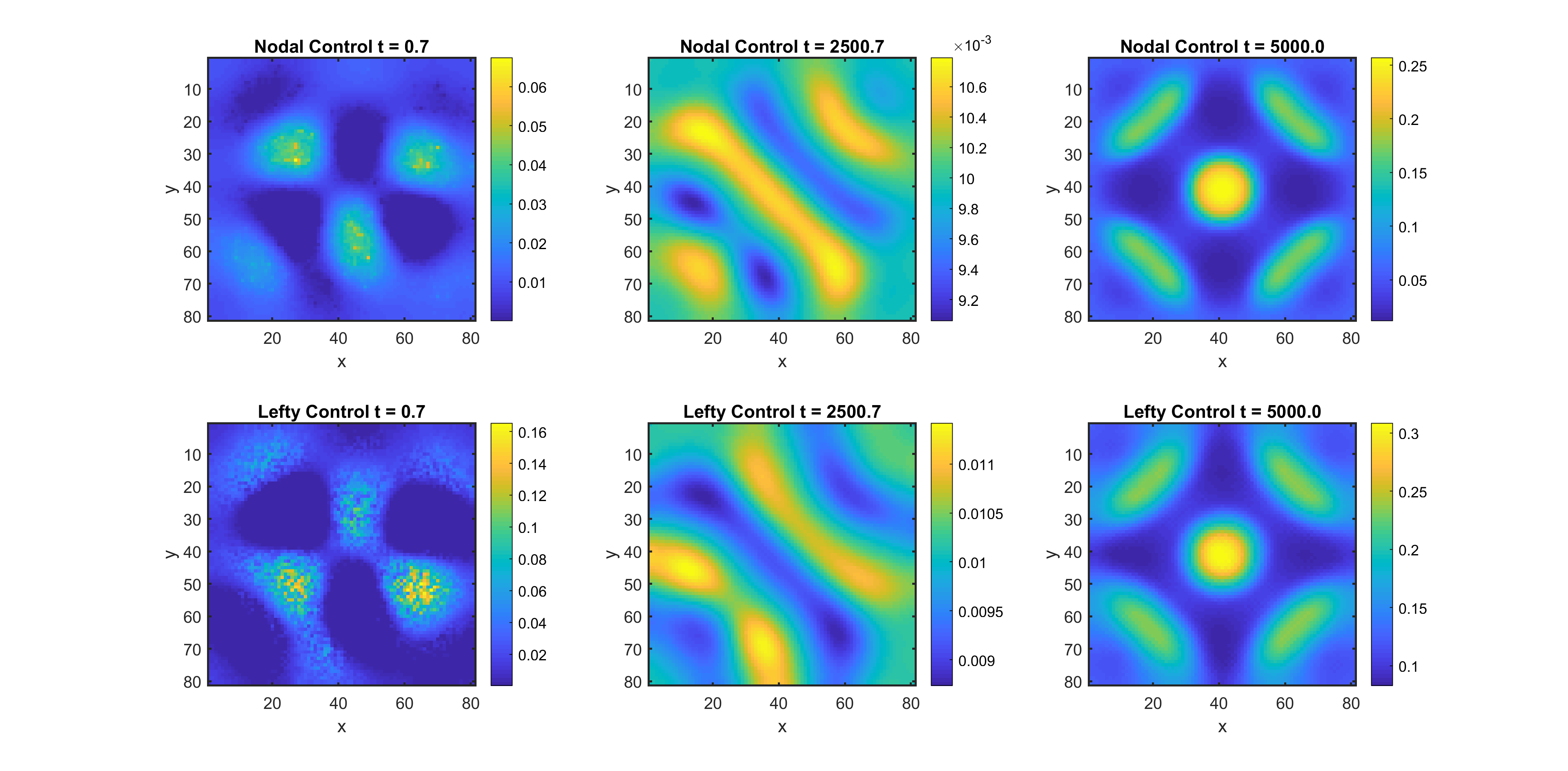}
    \caption{Case 4: Radial organization}
\end{subfigure}
\caption{Optimal control signals for Nodal (top) and Lefty (bottom) at representative time points. Time values shown are in units of 2 hours. Blue: low, yellow: high control intensity.}
\label{fig:control_signals}
\end{figure*}
\begin{table}[H]
    \centering
    \caption{Target configurations and achieved relative errors}
    \label{tab:combined_results}
    \scriptsize
    \begin{tabular}{ccccccc}
        \hline
        \textbf{Target} & $\boldsymbol{\alpha_n}$ & $\boldsymbol{\alpha_l}$ & \textbf{Pattern} & \textbf{Time (h)} & \textbf{Nodal Err.} & \textbf{Lefty Err.} \\
        \hline
        Initial & 0.8 & 4.0 & Striped & - & - & - \\
        Case 1 & 0.5 & 4 & Spotted & 2000 & 2.36e-2 & 1.90e-2 \\
        Case 2 & 0.8 & 4.5 & Spotted& 4000 & 1.18e-2 & 4.42e-3 \\
        Case 3 & 1 & 4.6 & Striped& 8000 & 1.42e-3 & 6.06e-4 \\
        Case 4 & 1.5 & 8.0 & Mixed& 10000 & 2.25e-3 & 1.00e-3 \\
        \hline
    \end{tabular}
\end{table}

The achieved control precision is quantified through relative errors $||y_{optimal}-y_{\Omega}||_2/||y_{\Omega}||_2$, with all cases showing errors below 2.5\%(Table \ref{tab:combined_results}). Figure \ref{fig:pattern_comparisons} displays the comparison between achieved solutions and target patterns for both species in each case.

The spatiotemporal control patterns (Fig. \ref{fig:control_signals}) reveal distinct regulatory strategies for each morphological transition. For Case 1, the controls selectively modulate production to fragment the initial stripe structure. Case 2 employs similar spatial distribution but with lower Lefty control amplitude. Case 3 systematically redirects patterns into aligned stripes. Case 4 combines directional stripe formation with radial organization. Temporally, each control progresses through destabilization and refinement phases.

\section{Conclusion}
\label{sec7}

This paper developed optimal control techniques for coupled reaction-diffusion systems that arise in synthetic biology. We proved that our control approach is mathematically sound and demonstrated its effectiveness on the Nodal-Lefty reaction-diffusion system, successfully transforming stripe patterns into spots and other shapes. Our spatiotemporal control framework advances the active manipulation of Turing patterns, enabling transitions between diverse patterns through computed external inputs. While the current optimization requires significant computation for high-resolution grids, future work will focus on developing efficient algorithms that maintain control accuracy at reduced computational cost. Additionally, we will explore nonlinear control strategies for improved efficiency and incorporate stochastic perturbations to better capture the inherent variability of biological systems.


\begin{thebibliography}{00}

\bibitem{turing} A. M. Turing, ``The chemical basis of morphogenesis,'' \textit{Phil. Trans. R. Soc. Lond. B}, vol. 237, no. 641, pp. 37--72, 1952.

\bibitem{barbu1993} V. Barbu, \textit{Analysis and Control of Nonlinear Infinite Dimensional Systems}, Academic Press, 1993.

\bibitem{karig} D. Karig, K. M. Martini, T. Lu, N. A. DeLateur, N. Goldenfeld, and R. Weiss, ``Stochastic Turing patterns in a synthetic bacterial population,'' \textit{Proc. Natl. Acad. Sci.}, vol. 115, no. 26, pp. 6572--6577, 2018.

\bibitem{sekine} R. Sekine, T. Shibata, and M. Ebisuya, ``Synthetic mammalian pattern formation driven by differential diffusivity of Nodal and Lefty,'' \textit{Nat. Commun.}, vol. 9, no. 1, Art. no. 5456, 2018.

\bibitem{tica} J. Tica, M. Oliver Huidobro, T. Zhu, G. K. A. Wachter, R. H. Pazuki, D. G. Bazzoli, N. S. Scholes, E. Tonello, H. Siebert, M. P. H. Stumpf, R. G. Endres, and M. Isalan, ``A three-node Turing gene circuit forms periodic spatial patterns in bacteria,'' \textit{Cell Syst.}, vol. 15, no. 12, pp. 1123--1132, 2024.

\bibitem{diambra} L. Diambra, V. R. Senthivel, D. B. Menendez, and M. Isalan, ``Cooperativity to increase Turing pattern space for synthetic biology,'' \textit{ACS Synth. Biol.}, vol. 4, no. 2, pp. 177--186, 2015.

\bibitem{Kennedy2010} M. J. Kennedy, R. M. Hughes, L. A. Peteya, J. W. Schwartz, M. D. Ehlers, and C. L. Tucker, ``Rapid blue-light-mediated induction of protein interactions in living cells,'' \textit{Nat. Methods}, vol. 7, no. 12, pp. 973--975, 2010.

\bibitem{Levskaya2009} A. Levskaya, O. D. Weiner, W. A. Lim, and C. A. Voigt, ``Spatiotemporal control of cell signalling using a light-switchable protein interaction,'' \textit{Nature}, vol. 461, no. 7266, pp. 997--1001, 2009.

\bibitem{muller2012differential} P. M\"uller, K. W. Rogers, B. M. Jordan, J. S. Lee, D. Robson, S. Ramanathan, and A. F. Schier, ``Differential diffusivity of Nodal and Lefty underlies a reaction-diffusion patterning system,'' \textit{Science}, vol. 336, no. 6082, pp. 721--724, 2012.

\bibitem{Sako2016} K. Sako, S. J. Pradhan, V. Barone, \'A. Ingl\'es-Prieto, P. M\"uller, V. Ruprecht, D. \v{C}apek, S. Galande, H. Janovjak, and C.-P. Heisenberg, ``Optogenetic control of Nodal signaling reveals a temporal pattern of Nodal signaling regulating cell fate specification during gastrulation,'' \textit{Cell Rep.}, vol. 16, no. 3, pp. 866--877, 2016.

\bibitem{McNamara2023} H. M. McNamara, B. Ramm, and J. E. Toettcher, ``Synthetic developmental biology: New tools to deconstruct and rebuild developmental systems,'' \textit{Semin. Cell Dev. Biol.}, vol. 141, pp. 33--42, 2023.

\bibitem{Lions1971} J.-L. Lions, \textit{Optimal Control of Systems Governed by Partial Differential Equations}, Springer-Verlag, 1971.

\bibitem{santos2019using} J. Santos-Moreno and Y. Schaerli, ``Using synthetic biology to engineer spatial patterns,'' \textit{Adv. Biosyst.}, vol. 3, no. 4, p. 1800280, 2019.

\bibitem{Troeltzsch2010} F. Tr\"oltzsch, \textit{Optimal Control of Partial Differential Equations}, American Mathematical Society, 2010.

\bibitem{Casas2018} E. Casas, C. Ryll, and F. Tr\"oltzsch, ``Optimal control of a class of reaction--diffusion systems,'' \textit{Comput. Optim. Appl.}, vol. 70, no. 3, pp. 677--707, 2018.

\bibitem{christofides2001} P. D. Christofides, \textit{Nonlinear and Robust Control of PDE Systems: Methods and Applications to Transport-Reaction Processes}, Birkh\"auser Boston, 2001.

\bibitem{pierre2010global} M. Pierre, ``Global existence in reaction-diffusion systems with control of mass: a survey,'' \textit{Milan J. Math.}, vol. 78, pp. 417--455, 2010.

\bibitem{Ryll2016} C. Ryll, ``Optimal control of patterns in some reaction--diffusion-systems,'' Ph.D. Thesis, \textit{Technische Universit\"at Berlin}, 2016.

\bibitem{clark1990} C. W. Clark, \textit{Mathematical Bioeconomics}, Wiley-Interscience, 1990.

\bibitem{ebrahimkhani2019synthetic} M. R. Ebrahimkhani and M. Ebisuya, ``Synthetic developmental biology: build and control multicellular systems,'' \textit{Curr. Opin. Chem. Biol.}, vol. 52, pp. 9--15, 2019.

\bibitem{polak1969note} E. Polak and G. Ribiere, ``Note sur la convergence de m\'ethodes de directions conjugu\'ees,'' \textit{Rev. Fr. Inform. Rech. Oper.}, vol. 3, no. 16, pp. 35--43, 1969.



\end{thebibliography}
\end{document}